\documentclass[12pt,reqno]{amsart}
\usepackage{tikz,dsfont}
\usepackage{amsmath, amssymb, graphicx, mathrsfs, hyperref}
\usepackage{verbatim} 
\usepackage{enumerate}
\newtheoremstyle{dotless}{}{}{\itshape}{}{\bfseries}{}{ }{} 
\theoremstyle{dotless}

\newtheorem{thm}{Theorem}

\begin{document}

\pagenumbering{arabic} \setcounter{page}{1}

\title{A large gap in a dilate of a set}

\author{George Shakan}

 \thanks{The author is supported by Ben Green's Simons Investigator Grant 376201 and also thanks him for introducing him to the question and useful discussions.}
\address{Department of Mathematics \\
University of Oxford }    
\email{shakan@maths.ox.ac.uk}

\begin{abstract}
Let $A \subset \mathbb{F}_p$ with $|A| > 1$. We show there is a $d \in \mathbb{F}_p^{\times}$ such that $d \cdot A$ contains a gap of size at least $2p/ |A| - 2 $. 
\end{abstract}

\maketitle

\section{Introduction}

Let $p$ be a prime and $A \subset \mathbb{F}_p$. We let $ g(A)$ be the largest gap in $A$, that is the largest intger $g(A)$ such that there is a $t \in \mathbb{F}_p$ satisfying
$$(\{1 , \ldots, g(A)\} + t) \cap A = \emptyset.$$
By the pigeon-hole principle, 
\begin{equation} \label{ph} g(A) \geq p / |A| - 1\end{equation}
For $d \in \mathbb{F}_p^{\times}$ we define
$$d \cdot A : = \{da : a \in A\}.$$
We seek lower bounds for 
$$L(A) : = \frac{|A|}{p} \sup_{d \in \mathbb{F}_p^{\times}} g(d \cdot A).$$
Note $L(A)$ is translation and dilation invariant. By \eqref{ph}, we have $L(A) \geq 1 - |A|/p$. Our goal is to double this bound.

\begin{thm}\label{main} Let $p$ be a prime and $A \subset \mathbb{F}_p$ with $|A| > 1$. Then
$$L(A) \geq 2(1 - \frac{|A| }{p} ),$$
or equivalently
$$\sup_{d \in \mathbb{F}_p^{\times}} g(d \cdot A) \geq 2(p/|A| - 1)  .$$
\end{thm}

We prove Theorem~\ref{main} using the polynomial method, or more precisely Redei's method \cite{Re}. Our work bears some similarity to the recent work \cite{D}. As asked by Ben Green \cite{BG}, it would be of interest to better understand $L(A)$, especially in the special case $|A| \sim \sqrt{p}$. In this case, is $L(A) \geq C$ for any fixed $C$? We remark that for $|A| \leq 1/100 \log p$ or $|A| \sim cp$, with $0 < c < 1$, one may improve upon Theorem~\ref{main} by applying Dirichlet's box principle and Szemer\'edi's theorem \cite{S}, respectively.

\section{Proof of Theorem~\ref{main}}

\begin{proof}[Proof of Theorem~\ref{main}]
We apply the polynomial method. We may suppose $|A| < p$ and set
$$m = \sup_{d \in \mathbb{F}_p^{\times}} g(d\cdot A) + 1.$$ 
Since $|A| > 1$, we have $m< p$. Let $B = \{1 , \ldots , m\}$ and 
$$P = A \times B \subset \mathbb{F}_p^2.$$
Thus the line $y = dx + t$ intersects $P$ for every $d \in \mathbb{F}_p^{\times}$ and $t \in \mathbb{F}_p$. 
 Let
$$k : = |A| |B| - p + 1,$$
and 
$$w(d,t) =   d \prod_{a\in A, b \in B} (b + da + t).$$
It follows that $w$ vanishes on $\mathbb{F}_p^2$ and so by \cite[Theorem 1]{Al}, we have 
$$w(d,t) =   (t^p - t) u(d,t) +(d^p - d)v(d,t),$$
for some $u,v$ of degree at most $k$. Taking the homogeneous part of degree $|B||A| + 1$ and setting $d = 1$, we find a $g$ and $h$ of degree at most $k$ such that 
$$f(t) := \prod_{a \in A} (t+ a)^{|B|} = t^pg(t) + h(t).$$
Then for every $a \in A$,
$$ (t+a)^{|B| - 1} | f'(t) , \ \ \ f'(t) = t^p g'(t) + h'(t) , $$
and so $(t + a)^{|B| - 1}$ must also divide
$$ (t^p g'(t) + h'(t))g(t) - (t^p g(t) + h(t))g'(t)  = h'(t) g(t) - h(t) g'(t).$$
In other words, $$\prod_{a \in A} (t + a)^{|B|  - 1},$$ divides a polynomial of degree at most $2k -1$ . We conclude $ (|B| - 1)|A| \leq  2k -1$, which implies
\begin{equation}\label{1st} p \leq k + |A|, \end{equation}
or 
\begin{equation}\label{2nd}h'(t) g(t) = h(t) g'(t) .\end{equation}
We claim if \eqref{2nd} holds with ${\rm deg }(g) , {\rm deg}(h) < p$, then $g(t)$ and $h(t)$ have the same roots (with multiplicity) in an algebraic closure of $\mathbb{F}_p$. Indeed, suppose that $(t+\alpha)^d | g(t)$ but $(t+\alpha)^{d+1} \nmid g(t)$ for some $\alpha \in \overline{\mathbb{F}_p}$.  Thus there is a $g_2 \in \overline{\mathbb{F}_p}[t] \setminus (t+\alpha)\overline{\mathbb{F}_p}[t] $ such that 
$$ g(t) = g_2(t)(t+\alpha)^d,  \ \ \ g'(t) = (t+\alpha)^{d-1}( dg_2(t) + (t+\alpha) g_2'(t)). $$
We have  $(t+\alpha)^{d-1} | g'(t)$ and since $d < p$, we also have $(t+\alpha)^{d} \nmid g'(t)$. Thus by \eqref{2nd} we find $(t+\alpha) | h(t)$. Then we may let $g(t) = (t+\alpha)g_1(t)$ and $h(t) = (t+\alpha) h_1(t)$ for some $g_1 , h_1 \in \overline{\mathbb{F}}_p[t]$. Substituting these into \eqref{2nd} and simplifying reveals
$$h_1'(t) g_1(t)(t+\alpha)^2 = h_1(t) g_1'(t)(t+\alpha)^2.$$ Thus \eqref{2nd} is satisfied for $g_1(t)$ and $h_1(t)$. The claim follows by induction on ${\rm deg}(g)$.


Thus if $k < p$ and \eqref{2nd} holds then $$ h(t) = c g(t) , \ \ \ c \in \mathbb{F}_p,$$
and so 
$$f(t) = (t^p + c) g(t) = (t+c)^p g(t). $$
Since $|B| = m < p$, this is impossible in light of 
$$f(t) = \prod_{a\in A} (t+a)^{|B|} .$$
Thus \eqref{1st} holds or $k \geq p$ and in either case we find $$|A||B| \geq 2p-|A| -1$$
and so 
$$\sup_{d \in \mathbb{F}_p^{\times}} g(d\cdot A)  \geq 2p/|A|  - 2-1/|A|.$$
We now remove the $1/|A|$, using that the left hand side is an integer. If $|A| = 2$, then one easily checks that Theorem~\ref{main} holds. Otherwise $2 < |A| < p$ and so $|A| \nmid 2p$. Thus $2p/|A| - 2$ is not an integer. Since the fractional part of $2p/|A| - 2$ is at least $1/|A|$, we have that 
$$\lceil 2p/|A| - 2 - 1/|A| \rceil \geq 2p/|A| - 2,$$
and so 
 $$\sup_{d \in \mathbb{F}_p^{\times}} g(d\cdot A)  \geq 2p/|A|  - 2.$$

\end{proof}

We remark the above proof fails for general point sets that are not necessarily cartesian products. Indeed, take a blocking set construction \cite[Page 107]{M} in $\mathbb{P}(\mathbb{F}_p^3)$ of size $\sim 3/2 p$. After projective transformation, make the line at infinity contain precisely one point, say $(1,0,0)$. Deleting this point creates a subset, $P$, of affine space such that every non-horizontal line intersects $P$. The proof does remain valid if we restrict to point sets $P$ such that $\pi(P)$ is small, where $\pi$ is a projection onto a coordinate axis.


\begin{thebibliography}{99}


\bibitem[Al99]{Al} A. Noga, Combinatorial Nullstellensatz. Combin. Probab. Comput. 8 (1999), no. 1-2, 7-29.


\bibitem[BSW20]{D} D. Di Benedetto, J. Solymosi, E. White, On the directions determined by a cartesian product in an affine Galois plane. Preprint arXiv:2001.06994. 

\bibitem[Gr20+]{BG} B. Green, 100 Open Problems, Preprint.


\bibitem[Mo07]{M} E. Moorhouse, Incidence Geometry, Course \href{ http://math.ucr.edu/home/baez/qg-fall2016/incidence_geometry.pdf}{http://math.ucr.edu/home/baez/qg-fall2016/incidence\_geometry.pdf}


\bibitem[Re73]{Re} L. R\'edei, Lacunary polynomials over finite fields, North Holland, Amsterdam, 1973.

 \bibitem[Sz75]{S} E. Szemer\'edi, On sets of integers containing no $k$ elements in arithmetic progression, Acta Arith. 27 (1975), 299-345.

\end{thebibliography}
\end{document}